\documentclass[12pt, 14paper,reqno]{amsart}
\usepackage{amsmath,amsfonts,amssymb}
\usepackage[breaklinks]{hyperref}
\usepackage{graphicx}
\usepackage{longtable}
\makeatletter
\@namedef{subjclassname@2020}{%
  \textup{2020} Mathematics Subject Classification}
\makeatother

\theoremstyle{plain}
\newtheorem{thm}{Theorem}[section]
\newtheorem{lem}{Lemma}[section]
\newtheorem{cor}{Corollary}[section]
\newtheorem{prop}{Proposition}[section]

\theoremstyle{proof}

\numberwithin{equation}{section}

\begin{document} 
\title[On a Lebesgue-Ramanujan-Nagell equation]{On the complete solutions of a generalized Lebesgue-Ramanujan-Nagell equation}
\author{Kalyan Chakraborty and Azizul Hoque}
\address{KC @Department of Mathematics, SRM University AP, Mangalagiri-Mandal, Guntur-522240, Andhra Pradesh, India}
\email{kalyan.c@srmap.edu.in} 
\address{AH @Department of Mathematics, Faculty of Science, Rangapara College, Rangapara, Sonitpur-784505, Assam, India.}
\email{ahoque.ms@gmail.com}
\date{\today}
\maketitle

\begin{abstract}
We consider the generalized Lebesgue-Ramanujan-Nagell equation $x^2+17^k41^\ell 59^m=2^\delta y^n$ in the unknown integers $x\geq 1, y>1,n\geq 3$ and $k, \ell, m\geq 0$ satisfying $\gcd(x,y)=1$. We first find  all the integer solutions of the above equation, and then use this result to determine all the integer solutions of some other Lebesgue-Ramanujan-Nagell type equations. Our method uses the classical results of Bilu, Hanrot and Voutier on existence of primitive divisors of Lehmer sequences in combination with number theoretic arguments and computer search.
\end{abstract}
\section{Introduction}\label{S1}
Let $d$ and $\lambda$ be two fixed positive integers. The Diophantine equation, 
\begin{equation}\label{eqi1}
x^2+d^m=\lambda y^n,~~ x,y\geq 1, m\geq 0, n\geq 3,
\end{equation} 
is known as a generalized Lebesgue-Ramanujan-Nagell equation. There are many results concerning the integer solutions $(x, y, m, n)$ of \eqref{eqi1} for given $d$ and $\lambda$. We refer the readers to the papers \cite{AA02, BP12,CO03,P2013} for $\lambda=1$; \cite{AL09, ZL12} for $\lambda=2$ and \cite{BHS19, CHS20, LTT09} for $\lambda=4$ for further reading. Some authors studied more general Diophantine equations (cf. \cite{AAL02, BU01, CH22, CHS21, HO20, LE93}).  We refer the beautiful survey \cite{LS20} for further information on the topic. 

In recent times many authors considered the following generalization of \eqref{eqi1}:
\begin{equation}\label{eqi2}
x^2+p_1^{k_1}p_2^{k_2}\cdots p_r^{k_r}=\lambda y^n,~~ x,y\geq 1, \gcd(x,y)=1, k_1, k_2,\cdots, k_r, m\geq 0, n\geq 3,
\end{equation}
where $\lambda=1,2$ and $p_1, p_2, \cdots, p_r$ are distinct primes with $r\geq 2$. We refer the papers (\cite{CHS21b, DE17,LT09, AL08, PI07, ZL15} for $\lambda=1$; \cite{APA2024, AL09, PI04} for $\lambda=2$) for further information. In this paper, we continue the investigation on the integer solutions of certain equations of the form  \eqref{eqi2} with $\lambda=2^\delta$, where $\delta\geq 0$ is a fixed integer. More precisely, we consider the Diophantine equation 
$$x^2+17^k41^\ell 59^m =2^\delta y^n,~~x\geq 1, y>1, \gcd(x,y)=1, k, \ell, m, \delta\geq 0, n\geq 3,$$
in unknown integers $x,y,k,\ell, m,n$. Clearly, reducing the above equation at modulo $8$ for $\delta\geq 3$, we can conclude that it has no integer solutions. Thus, the above equation becomes non-trivial for $\delta=0,1,2$, and we rewrite it as follows:
\begin{equation}\label{eqn1}
x^2+17^k41^\ell 59^m =\lambda y^n,~~x\geq 1, y>1, \gcd(x,y)=1, k, \ell, m\geq 0, n\geq 3,
\end{equation}
where $\lambda\in \{1,2, 4\}$. Here, we completely solve \eqref{eqn1} for integers $x,y,k, \ell, m$ and $n$. More precisely, we prove:
\begin{thm}\label{thm}
The equation \eqref{eqn1} has no solution except for:  
\begin{itemize}
\item[(i)] $n=3$, all the solutions are given by Table \ref{Tp3};
\item[(ii)] $n=4$, all the solutions are given by Table \ref{Tp4};
\item[(iii)] $n=5$, the solution is $(x,y,\lambda,k,\ell,m)=(38,5,1,0,2,0)$.
\end{itemize}
\end{thm}

We organize the paper as follows. In \S\ref{S2}, we recall some important ingredients which are used in the proof of Theorem \ref{thm}. We present the proof of Theorem \ref{thm} in \S\ref{S3}, which is divided into several subsections. In \S\ref{proofpropm3}, we treat the case $3\mid n$. In this case, we transform \eqref{eqn1} into several elliptic equations written in cubic models for which we need to determine all $\{17,41,59\}$-integer points. Recall that if $S$ is a finite set of prime numbers, then an $S$-integer is  a rational number $r/s$ with coprime integers $r$ and $s>0$ such that the prime factors of $s$ lie in $S$. We utilize the same technique in \S\ref{proofpropm4} to handle \eqref{eqn1} when $n$ is a multiple of $4$; however, in this case, we transform \eqref{eqn1} into quartic curves. We handle \eqref{eqn1} for prime $n \geq  5$ in \S\ref{proofpropp}. Here, we apply a method that uses primitive divisor theorem for Lehmer sequences. In this case too, we are able to obtain several elliptic equations written in cubic models for which we find all their $\{17, 41, 59\}$-integer points. We finally summerize the proof of Theorem \ref{thm} in \S\ref{proofthm}. In \S\ref{S4}, we discuss some existing results   and derive some corollaries. All the computations are done with MAGMA \cite{BCP97}. 

\section{Preliminaries}\label{S2}
We begin this section with the following lemma which follows from \cite[Corollary 3.1]{YU05}.

\begin{lem}\label{lemYU}
Let $d$ be a square-free positive integer and $h(-d)$ denotes the class number of $\mathbb{Q}(\sqrt{-d})$. For any odd integer $n\geq 3$ coprime to $h(-d)$, all integer solutions $(X,Y,Z)$ of the equation 
\begin{equation*}\label{eqYU}
X^2+dY^2=\lambda Z^n,~~ X,Y\geq 1, \gcd(X, dY)=1, \lambda=1,2,4,
\end{equation*} 
can be expressed as 
$$\frac{X+Y\sqrt{-d}}{\sqrt{\lambda}}=\varepsilon_1\left(\frac{a+\varepsilon_2 b\sqrt{-d}}{\sqrt{\lambda}}\right)^n,$$
where $\varepsilon_1,\varepsilon_2\in\{-1, 1\}$, and $a$ and $b$ are positive integers satisfying $\lambda Z=a^2+b^2d$ and $\gcd(a, bd)=1$.
\end{lem}

Let $\alpha$ and $\beta$ be algebraic integers such that $(\alpha + \beta)^2$ and $\alpha\beta$ are non-zero coprime rational integers as well as $\alpha/\beta$ is not a root of unity. Then the pair $(\alpha, \beta)$ is called Lehmer pair, and two Lehmer pairs $(\alpha_1, \beta_1)$ and $(\alpha_2, \beta_2)$ are said to be equivalent if $\alpha_1/\alpha_2=\beta_1/\beta_2\in\{\pm 1, \pm\sqrt{-1}\}$. For any Lehmer pair $(\alpha, \beta)$, the Lehmer sequence is defined by  
$$\mathcal{L}_n(\alpha, \beta)=\begin{cases}
\dfrac{\alpha^n-\beta^n}{\alpha-\beta} & \text{ if } 2\nmid n, \vspace{1mm}\\
\dfrac{\alpha^n-\beta^n}{\alpha^2-\beta^2} & \text{ if } 2\mid n.
\end{cases}$$
A prime divisor $p$ of $\mathcal{L}_n(\alpha, \beta)$ is primitive if  $p\nmid(\alpha^2-\beta^2)^2
\mathcal{L}_1(\alpha, \beta) \mathcal{L}_2(\alpha, \beta) \cdots \mathcal{L}_{n-1}(\alpha, \beta)$. 
Note that $\left( (\alpha+\beta)^2, (\alpha-\beta)^2\right)$ is the parameter of the Lehmer pair $(\alpha, \beta)$. 
We extract the following classical result from Bilu et al. \cite[Theorem 1.4 and Tables 2, 4]{BH01}. 
\begin{lem}\label{lemBH}
For any prime $p>13$, the Lehmer numbers $\mathcal{L}_p (\alpha, \beta) $ have primitive divisors. Further, if $p$ is a prime such that $7\leq p\leq 13$ and the Lehmer numbers $\mathcal{L}_p(\alpha, \beta)$ have no primitive divisors, then up to equivalence, the parameters $(A, B)$ of the corresponding Lehmer pair $(\alpha, \beta)$ are given by  
\begin{itemize}
\item[(i)] for $q=7, (A, B)=(1,-7), (1, -19), (3, -5), (5, -7), (13, -3), (14, -22)$;  
\item[(ii)] for $q=13, (A, B)=(1,-7)$.
\end{itemize}
\end{lem}
We derive the following lemma from \cite{LE30,WA55}.
\begin{lem}\label{lemLE}
If $p$ is a primitive prime divisor of a Lehmer number $L_n(\alpha, \beta)$, then $p\equiv \pm 1 \pmod{n}$. Further, if $D=(\alpha^2-\beta^2)^2$ then for any odd prime $n$, $p\equiv \left(\frac{D}{p}\right)\pmod n$, where $\left(\frac{D}{p}\right)$
denotes the Legendre symbol.
\end{lem}

\section{Proof of Theorem \ref{thm}}\label{S3}
We divide this section in three subsections to make the proof explicit.  
\subsection{When $n$ is a multiple of $3$}\label{proofpropm3}
Assume that $n=3N$ for some integer $N\geq 1$. Then \eqref{eqn1} can be written as
\begin{equation}\label{eqm31}
x^2+17^k41^\ell 59^m =\lambda y^{3N},~~x\geq 1, y>1, \gcd(x,y)=1, k, \ell, m\geq 0, N\geq 1,
\end{equation}
where $\lambda\in \{1,2, 4\}$.

\begin{prop}\label{propm3}
All integer solutions of \eqref{eqm31} are listed in Table \ref{Tp3}.  
\end{prop}

\begin{proof}
Let $k=6k_1+k_2, \ell=6\ell_1+\ell_2$ and $m=6m_1+m_2$ for some integers $k_1, \ell_1, m_1\geq 0$ and $k_2, \ell_2, m_2\in\{0,1,2,3,4,5\}$. Then 
\begin{equation}\label{eqm32}
\lambda^2 17^k 41^\ell 59^m=Dz^6,
\end{equation}
where 
\begin{equation}\label{eqm33}
\begin{cases}
D=\lambda^2 17^{k_2}41^{\ell_2}59^{m_2}\\
 z=17^{k_1}41^{\ell_1}59^{m_1}.
 \end{cases}
\end{equation}
We now multiply \eqref{eqm31} by $\lambda^2$ and then utilize \eqref{eqm32} to get
$$
(\lambda x)^2+Dz^6=(\lambda y^N)^3.
$$
Dividing both sides of the above equality by $z^6$, we get 
\begin{equation}\label{eqm34}
X^2=Y^3-D,
\end{equation}
where 
\begin{equation}\label{eqm35}
\begin{cases}
X=\dfrac{\lambda x}{z^3}\vspace{1mm}\\
 Y=\dfrac{\lambda y^N}{z^2}.
 \end{cases}
\end{equation}
Now the problem of finding the integer solutions of \eqref{eqm31} is reduced to finding $\{17,41,59\}$-integer points on the $648$ elliptic curves defined by \eqref{eqm34}. We use \texttt{SIntegralPoints} subroutine of MAGMA (V2.25-5)\cite{BCP97} to compute all $\{17,41,59\}$-integer points on these elliptic curves. Note that we eliminate $\{17,41,59\}$-integers points $(X,Y)$ with $XY=0$ as they yield to $xy = 0$. We also eliminate $(X,Y)$ satisfying $\lambda\nmid XY$ by \eqref{eqm35}. Finally taking into account that $\gcd(x,y)=1$,  we don't consider $\{17,41,59\}$-integers points $(X,Y)$ when $X/\lambda$ and $Y/\lambda$ are not coprimes. The rest of $\{17,41,59\}$-integer points and the corresponding $D$ are listed in Table \ref{ts-3}. 

\begin{center}
\begin{longtable}{ccc| ccc}
\caption{\tiny$S$-integral points $(X,Y)$ on \eqref{eqm34} corresponding to $D$'s} \label{ts-3} \\

\hline \multicolumn{1}{c}{$X$} & \multicolumn{1}{c}{$Y$} & \multicolumn{1}{c|}{$D$}& \multicolumn{1}{c}{$X$} & \multicolumn{1}{c}{$Y$}& \multicolumn{1}{c}{$D$}  \\ \hline 
\endfirsthead

\multicolumn{6}{c}%
{{\bfseries \tablename\ \thetable{} -- continued from previous page}} \\
\hline \multicolumn{1}{c}{$X$} & \multicolumn{1}{c}{$Y$} & \multicolumn{1}{c}{$D$}& \multicolumn{1}{c}{$X$} & \multicolumn{1}{c}{$Y$}& \multicolumn{1}{c}{$D$} \\ \hline 
\endhead
\hline \multicolumn{6}{c}{{Continued on next page}} \\ \hline
\endfoot
\hline 
\endlastfoot
5220&307&$17\cdot 41^2\cdot59$& 
$2\cdot245$&$2\cdot33$&$2^2\cdot17^2\cdot41$\\ \hline
$4\cdot7$&$4\cdot3$&$4^2\cdot59$ & $4\cdot21$&$4\cdot5$&$4^2\cdot 59$\\ \hline
$4\cdot393$& $4\cdot35$& $4^2\cdot17^2\cdot59$&$4\cdot525$& $4\cdot41$& $4^2\cdot17^2\cdot59$\\ \hline $4\cdot2389$&$4\cdot951$&$4^2\cdot17^5\cdot41\cdot59$&$4\cdot5195$&$4\cdot 189$&$4^2\cdot17^2\cdot 59$\\ \hline 
$4\cdot10535$& $4\cdot411$& $4^2\cdot41^4\cdot59$& $4\cdot28735$&$4\cdot591$&$4^2\cdot59$\\ \hline
$4\cdot55049$&$4\cdot945$&$4^2\cdot41^2\cdot59^3$&$4\cdot155963$&$4\cdot1827$&$4^2\cdot17\cdot41^3\cdot59$\\ \hline
$4\cdot791561$&$4\cdot5391$&$4^2\cdot17\cdot41\cdot59^3$&
$4\cdot2834943$&$4\cdot12629$&$4^2\cdot17^3\cdot41^3\cdot59$\\ \hline
$2\cdot36785$&$2\cdot2469$&$2^2\cdot17^3\cdot41^2\cdot59^2$&
$\frac{2\cdot721267}{17^3}$&$\frac{2\cdot6153}{17^2}$&$4^2\cdot17^2\cdot59$
 \end{longtable}
\end{center}
\vspace{-10mm}
We now utilize the relations \eqref{eqm32}, \eqref{eqm33} and \eqref{eqm35} to find the solutions of \eqref{eqm31}, which are listed in Table \ref{Tp3}. This completes the proof. 
\end{proof}

\begin{center}
\begin{longtable}{ccccccc| ccccccc}
\caption{\small Solutions of  \eqref{eqm31}} \label{Tp3} \\

\hline \multicolumn{1}{c}{$x$} & \multicolumn{1}{c}{$y$} & \multicolumn{1}{c}{$\lambda$}& \multicolumn{1}{c}{$k$} & \multicolumn{1}{c}{$\ell$}& \multicolumn{1}{c}{$m$} & \multicolumn{1}{c|}{$N$} & \multicolumn{1}{c}{$x$} & \multicolumn{1}{c}{$y$} & \multicolumn{1}{c}{$\lambda$}& \multicolumn{1}{c}{$k$} & \multicolumn{1}{c}{$\ell$}& \multicolumn{1}{c}{$m$} & \multicolumn{1}{c}{$N$} \\ \hline 
\endfirsthead

\multicolumn{14}{c}%
{{\bfseries \tablename\ \thetable{} -- continued from previous page}} \\
\hline \multicolumn{1}{c}{$x$} & \multicolumn{1}{c}{$y$} & \multicolumn{1}{c}{$\lambda$}& \multicolumn{1}{c}{$k$} & \multicolumn{1}{c}{$\ell$}& \multicolumn{1}{c}{$m$} & \multicolumn{1}{c|}{$N$}& \multicolumn{1}{c}{$x$} & \multicolumn{1}{c}{$y$} & \multicolumn{1}{c}{$\lambda$}& \multicolumn{1}{c}{$k$} & \multicolumn{1}{c}{$\ell$}& \multicolumn{1}{c}{$m$} & \multicolumn{1}{c}{$N$}
 \\ \hline 
\endhead
\hline \multicolumn{14}{c}{{Continued on next page}} \\ \hline
\endfoot
\hline 
\endlastfoot
5220&307&1&1&2&1&1& 245&33& 2& 2& 1& 0& 1\\ \hline
7&3&4& 0&0&1&1&  21&5&4&  0&0&1&1 \\ \hline
393& 35& 4& 2& 0& 1& 1&525& 41& 4& 2& 0& 1& 1\\ \hline

2389&951&4&5&1&1&1&5195&189&4&2&0&1&1\\ \hline
10535&411&4&0&4&1&1&
28735&591&4&0&0&1&1\\ \hline 
55049&945&4&0&2&3&1&155963&1827&4&1&3&1&1\\ \hline
791561&5391&4&1&1&3&1&2834943&12629&4&3&3&1&1\\ \hline
36785&2469&2&3&2&2&1&721267&6153&4&8&0&1&1 
 \end{longtable}
\end{center}
\vspace{-10mm}
\subsection{When $n$ is a multiple of $4$}\label{proofpropm4}
Let $n=4t$ for some integer $t\geq 1$. Then \eqref{eqn1} becomes
\begin{equation}\label{eqm41}
x^2+17^k41^\ell 59^m=\lambda\left(y^t\right)^4,~ x\geq 1, y>1, \gcd(x, y)=1, k,\ell, m\geq 0, t\geq 1, \lambda=1,2,4.\
\end{equation}

\begin{prop}\label{propm4}
All the solutions of  \eqref{eqm41} are listed in Table \ref{Tp4}.
\end{prop}

\begin{proof}
Assume that $k=4k_1+k_2, \ell=4\ell_1+\ell_2$ and $m=2m_1+m_2$, where $k_1,\ell_1, m_1\geq 0$ are integers and $k_2,\ell_2, m_2\in\{0,1,2,3\}$. Then \eqref{eqm41} can be written as
$$x^2+17^{k_2}41^{\ell_2}59^{m_2}(17^{k_1}41^{\ell_1}59^{m_1})^4=\lambda\left(y^t\right)^4.$$\vspace{1mm}
We put $X:=\dfrac{x}{(17^{k_1}41^{\ell_1}59^{m_1})^2}$ and $Y:=\dfrac{y^t}{17^{k_1}41^{\ell_1}59^{m_1}}$ in the above equation to get
\begin{equation}\label{eqxx}
X^2=\lambda Y^4-17^{k_2}41^{\ell_2}59^{m_2}.
\end{equation}
This defines $192$ quartic curves. We use \texttt{SIntegralLjunggrenPoints} subroutine of Magma to determine all $\{17,41,59\}$-integer points on these curves. Taking $XY\ne 0$ into account, we list the values of $(X,Y,\lambda, k_2, \ell_2, m_2)$ in Table \ref{Tn4}.

\begin{center}
\begin{longtable}{cccccc| cccccc}
\caption{\small Solutions of  \eqref{eqxx}} \label{Tn4} \\

\hline \multicolumn{1}{c}{$X$} & \multicolumn{1}{c}{$Y$} & \multicolumn{1}{c}{$\lambda$}& \multicolumn{1}{c}{$k_2$} & \multicolumn{1}{c}{$\ell_2$}& \multicolumn{1}{c|}{$m_2$} & \multicolumn{1}{c}{$X$} & \multicolumn{1}{c}{$Y$} & \multicolumn{1}{c}{$\lambda$}& \multicolumn{1}{c}{$k_2$} & \multicolumn{1}{c}{$\ell_2$}& \multicolumn{1}{c}{$m_2$} \\ \hline 
\endfirsthead

\multicolumn{12}{c}%
{{\bfseries \tablename\ \thetable{} -- continued from previous page}} \\
\hline \multicolumn{1}{c}{$X$} & \multicolumn{1}{c}{$Y$} & \multicolumn{1}{c}{$\lambda$}& \multicolumn{1}{c}{$k_2$} & \multicolumn{1}{c}{$\ell_2$}& \multicolumn{1}{c|}{$m_2$} & \multicolumn{1}{c}{$X$} & \multicolumn{1}{c}{$Y$} & \multicolumn{1}{c}{$\lambda$}& \multicolumn{1}{c}{$k_2$} & \multicolumn{1}{c}{$\ell_2$}& \multicolumn{1}{c}{$m_2$}  \\ \hline 
\endhead
\hline \multicolumn{12}{c}{{Continued on next page}} \\ \hline
\endfoot
\hline 
\endlastfoot
840&29&1&0&2&0&8&3&1&1&0&0\\ \hline
239&13&2&0&0&0&11&3&2&0&1&0 \\ \hline
69&7&2&0&1&0&171&11&2&0&1&0\\ \hline
31&5&2&2&0&0&299&15&2&2&1&0\\ \hline
9&5&4&0&1&1
 \end{longtable}
\end{center}
\vspace{-8mm}
We now use the expressions for $X$ and $Y$ to find the solutions of \eqref{eqm41} from Table \ref{Tn4}. We list all these solutions in Table \ref{Tp4}. This completes the proof. 
\end{proof}

\begin{center}
\begin{longtable}{ccccccc| ccccccc}
\caption{\small Solutions of  \eqref{eqm41}} \label{Tp4} \\

\hline \multicolumn{1}{c}{$x$} & \multicolumn{1}{c}{$y$} & \multicolumn{1}{c}{$\lambda$}& \multicolumn{1}{c}{$k$} & \multicolumn{1}{c}{$\ell$}& \multicolumn{1}{c}{$m$} & \multicolumn{1}{c|}{$t$} & \multicolumn{1}{c}{$x$} & \multicolumn{1}{c}{$y$} & \multicolumn{1}{c}{$\lambda$}& \multicolumn{1}{c}{$k$} & \multicolumn{1}{c}{$\ell$}& \multicolumn{1}{c}{$m$} & \multicolumn{1}{c}{$t$} \\ \hline 
\endfirsthead

\multicolumn{14}{c}%
{{\bfseries \tablename\ \thetable{} -- continued from previous page}} \\
\hline \multicolumn{1}{c}{$x$} & \multicolumn{1}{c}{$y$} & \multicolumn{1}{c}{$\lambda$}& \multicolumn{1}{c}{$k$} & \multicolumn{1}{c}{$\ell$}& \multicolumn{1}{c}{$m$} & \multicolumn{1}{c|}{$t$}& \multicolumn{1}{c}{$x$} & \multicolumn{1}{c}{$y$} & \multicolumn{1}{c}{$\lambda$}& \multicolumn{1}{c}{$k$} & \multicolumn{1}{c}{$\ell$}& \multicolumn{1}{c}{$m$} & \multicolumn{1}{c}{$t$}
 \\ \hline 
\endhead
\hline \multicolumn{14}{c}{{Continued on next page}} \\ \hline
\endfoot
\hline 
\endlastfoot
840&29&1&0&2&0&1&8&3&1&1&0&0&1\\ \hline
239&13&2&0&0&0&1&11&3&2&0&1& 0&1\\ \hline
69&7&2&0&1&0&1&171&11&2&0&1&0&1 \\ \hline
31&5&2&2&0&0&1&299&15&2&2&1&0&1 \\ \hline
9&5&4&0&1&1&1
 \end{longtable}
\end{center}
\vspace{-10mm}
\subsection{When $n\geq 5$ is prime}\label{proofpropp}
Here, we replace $n$ by $p$ in \eqref{eqn1} to assert that the exponent is prime, that is,
\begin{equation}\label{eqp1}
x^2+17^k41^\ell 59^m =\lambda y^p,~~x\geq 1, y>1, \gcd(x,y)=1, k, \ell, m\geq 0, p\geq 5,
\end{equation}
where $\lambda =1, 2,4$. 
\begin{prop}\label{propp}The only integer solution of \eqref{eqp1} is $(x,y,\lambda,k,\ell,m,p)=(38,5,1,0,2,0,5)$.

\end{prop}

\begin{proof}
Suppose that $k=2k_1+k_2, \ell=2\ell_1+\ell_2$ and $m=2m_1+m_2$, where $k_1,\ell_1, m_1\geq 0$ are integers and $k_2,\ell_2,m_2\in\{0,1\}$. Then 
$17^k41^\ell 59^m=dz^2$, where $d=17^{k_2}41^{\ell_2} 59^{m_2}$ and $z=17^{k_1}41^{\ell_1} 59^{m_1}$. 
Therefore \eqref{eqp1} becomes
$$x^2+dz^2=\lambda y^p, ~~x\geq 1, y>1, \gcd(x,dz)=1, k, \ell, m\geq 0, p\geq 5, \lambda =1, 2,4.$$
Using MAGMA, we see that $h(-d)\in\{1,3,4,8,72\}$, where $h(-d)$ denotes the class number of $\mathbb{Q}(\sqrt{-d})$.
Thus by Lemma \ref{lemYU}, we have
\begin{equation}\label{eqp2}
\frac{x+z\sqrt{-d}}{\sqrt{\lambda}}=\varepsilon_1\left(\frac{a+\varepsilon_2 b\sqrt{-d}}{\sqrt{\lambda}}\right)^p,
\end{equation} 
where $\varepsilon_1, \varepsilon_2\in\{-1, 1\}$, and $a$ and $b$ are positive integers satisfying
\begin{equation}\label{eqp3}
a^2+b^2d=\lambda y.
\end{equation}
Clearly, $b$ is odd. Since $y$ is odd, so that $a$ is even ( resp. odd) when $\lambda =1$ (resp. $\lambda = 2, 4$). 

Assume that $$\alpha:=\frac{a+\varepsilon_2 b\sqrt{-d}}{\sqrt{\lambda}},~~\beta:=\frac{a-\varepsilon_2 b\sqrt{-d}}{\sqrt{\lambda}}.$$
Clearly, both $\alpha$ and $\beta$ are algebraic integers satisfying $\gcd((\alpha+\beta)^2, \alpha\beta)=1$. Also, $\alpha/\beta$ is a root of $\lambda yZ^2-2(a^2-b^2d)+\lambda y=0$, and hence $\alpha/\beta$ is not a root of unity. Therefore, $(\alpha, \beta)$ is  Lehmer pair with parameters $(4a^2/\lambda, -4b^2d/\lambda)$. 

If $\mathcal{L}_n$ is the $n$-th Lehmer number for the pair $(\alpha, \beta)$, then 
\begin{equation}\label{eqp4}
|\mathcal{L}_p(\alpha, \beta)|=\frac{z}{b}=\frac{17^{k_1}41^{\ell_1} 59^{m_1}}{b}.
\end{equation}
Let $q$ be a primitive prime divisor of $\mathcal{L}(\alpha, \beta)$. Then by \eqref{eqp4}, $q\in\{17,41,59\}$. Utilizing Lemma \ref{lemLE}, we see that $q\ne 17$ since $p\geq 5$. Also, 
$$q=\begin{cases}
41 & \text{ when } p=5,7;\\
59 & \text{ when } p=5, 29.
\end{cases}
$$
We first consider $q=59$. Then by the definition of primitive divisor, $q\nmid (\alpha^2-\beta^2)^2=-16a^2b^2d/\lambda^2$. This gives that  $(k_2, \ell_2, m_2)\in\{(0,0,0), (0,1,0), (1,0,0), (1,1,0)\}$, which implies $d\in\{1,41,17, 697\}$. Again since $\left(\frac{-4b^2d}{59}\right)=\left(\frac{-d}{59}\right)=-1$, so that $59\equiv 1\pmod p$ implies $p=5$.

If $q=41$, then as before we get  $(k_2, \ell_2, m_2)\in\{(0,0,0), (0,0,1), (1,0,0), (1,0,1)\}$ and thus $d\in\{1,59, 17, 1003\}$. Again since $\left(\frac{-4b^2d}{41}\right)=\left(\frac{-d}{41}\right)=1$ or $-1$ according as $d=1,59$ or $d=17, 1003$, so that $41\equiv \pm 1\pmod p$ gives $(p, d)\in\{(7,17),(7,1003), (5,1), (5,59)\}$. We now divide the proof into several cases depending on the values of $p$ and $q$.
\subsection*{Case I:  $p>7$} 
From the above, it follows that $\mathcal{L}_p(\alpha, \beta)$ has no primitive divisors for $p>7$. This contradicts to the primitive divisor theorem for Lehmer sequences, which states that, if $p\geq 3$, then $\mathcal{L}_p(\alpha, \beta)$ has a primitive prime divisor with a few exceptional pairs $(\alpha, \beta)$. For any prime $p> 7$, these exceptional pairs are given by Lemma \ref{lemBH} in terms of their parameters. Therefore using Lemma \ref{lemBH} we get $(p,4a^2/\lambda, 4b^2d/\lambda)=(13,1,7)$, which further implies that $d=7$. This is not possible as the only prime divisors of $d$ are $17, 41, 59$. Therefore, \eqref{eqp1} has no integer solutions for $p>7$. 
\subsection*{Case II:  $p=7$} 
In this case, the only primitive divisor of $\mathcal{L}_7(\alpha, \beta)$ is $41$, and accordingly  $d=17, 1003$. Equating the imaginary parts in \eqref{eqp2}, we get 
\begin{equation}\label{eqp5}
\lambda^317^{k_1}41^{\ell_1}59^{m_1}=\varepsilon b(7a^6-35a^4b^2d+21a^2b^4d^2-b^6d^3),
\end{equation}
where $\varepsilon=\varepsilon_1\varepsilon_2=\pm 1$ and $d=17, 1003$. Also by the definition of primitive divisor, $41\nmid (\alpha^2-\beta^2)^2=-16a^2b^2d/\lambda$, we have $b\ne 41$. Since $b$ is odd and $\gcd(a, b)=1$, so that \eqref{eqp5} gives 
$$b=1,17^{k_1},59^{m_1}, 17^{k_1}59^{m_1}.$$

We first consider $b=1$. Then \eqref{eqp5} reduces to
$$\lambda^317^{k_1}41^{\ell_1}59^{m_1}=\varepsilon (7a^6-35a^4d+21a^2d^2-d^3).$$
 Assume that $k_1=2k_4+k_3, \ell_1=2\ell_4+\ell_3, m_1=2m_4+m_3$ with $k_3, \ell_3, m_3\in\{0, 1\}$ and then set $D:=\varepsilon\lambda 17^{k_3}41^{\ell_3}59^{m_3}$ to get
$$D(\lambda 17^{k_4}41^{\ell_4}59^{m_4})^2=7a^6-35a^4d+21a^2d^2-d^3.$$
Multiplying both sides by $7^2D^3$, and then putting $(X, Y):=(7Da^2, 7D^2\lambda 17^{k_4}41^{\ell_4}59^{m_4})$, we arrive at the following:
$$Y^2=X^3-35dDX^2+147d^2D^2X-49d^3D^3.$$
This defines $96$ elliptic curves, where we use \texttt{IntegralPoints} subroutine of MAGMA to compute all integer points. We first remove all integer points $(X,Y)$ satisfying $XY=0$ or $7\nmid X$ since $abd\ne0$ and $7\mid \gcd(X,Y)$. We then check $a^2=X/{7D}$ for the remaining points $(X,Y)$; but none of these shows that $a$ is an integer. Thus, none of these integer points lead to an integer solutions of \eqref{eqp1}. 

Assume that $b=17^{k_1}$. Then \eqref{eqp5} reduces to
$$\lambda^341^{\ell_1}59^{m_1}=\varepsilon (7a^6-35a^4b^2d+21a^2b^4d^2-b^6d^3).$$
As before, we consider $D_1:=\varepsilon\lambda 41^{\ell_3}59^{m_3}$ for $ \ell_3, m_3\in\{0, 1\}$. Then the above equation reduces to
$$D_1(\lambda 41^{\ell_4}59^{m_4})^2=7a^6-35a^4b^2d+21a^2b^4d^2-b^6d^3,$$
 where $\ell_1=2\ell_4+\ell_3, m_1=2m_4+m_3$.
We divide both sides by $b^6$ and then multiply by $7^2D_1^3$ to get
$$Y^2=X^3-35dD_1X^2+147d^2D_1^2X-49d^3D_1^3,$$
where $X=7D_1a^2/b^2$ and $Y=7D_1\lambda 41^{\ell_4}59^{m_4}/b^3$.
As in the previous case, we use \texttt{IntegralPoints} subroutine of MAGMA to compute all $\{17\}$-integral points on the above $48$ elliptic curves. We eliminate all integer points $(X,Y)$ satisfying $XY=0$ or $7\nmid X$ since $abd\ne0$ and $7\mid \gcd(X,Y)$. For the remaining points $(X,Y)$, we examine $a^2=Xb^2/{7D_1}$. However, none of these gives integer value to $a$, and hence none of these integer points lead to an integer solution of \eqref{eqp1}. 

Let $b=59^{m_1}$. Then \eqref{eqp5} becomes
$$\lambda^317^{k_1}41^{\ell_1}=\varepsilon (7a^6-35a^4b^2d+21a^2b^4d^2-b^6d^3).$$
As before, put $D_2:=\varepsilon\lambda 17^{k_3}41^{\ell_3}$ for $k_3, \ell_3, m_3\in\{0, 1\}$. Then we have
$$D_2(\lambda 17^{k_4}41^{\ell_4})^2=7a^6-35a^4b^2d+21a^2b^4d^2-b^6d^3,$$
 where $k_1=2k_4+k_3, \ell_1=2\ell_4+\ell_3$.
We divide both sides by $b^6$ and then multiply by $7^2D_1^3$ to get
$$Y^2=X^3-35dD_2X^2+147d^2D_2^2X-49d^3D_2^3,$$
where $X=7D_2a^2/b^2$ and $Y=7D_2\lambda 17^{k_4}41^{\ell_4}/b^3$.
In this case too, we use \texttt{IntegralPoints} subroutine of MAGMA to compute all $\{59\}$-integral points on the above $48$ elliptic curves.  As in the last case, we see that one  none of these integer points lead to an integer solution of \eqref{eqp1}.

Finally consider $b=17^{k_1}59^{m_1}$. Then \eqref{eqp5} reduces to
$$\lambda^341^{\ell_1}=\varepsilon (7a^6-35a^4b^2d+21a^2b^4d^2-b^6d^3).$$
As before, putting $D_3:=\varepsilon\lambda 41^{\ell_3}$ for $\ell_3\in\{0, 1\}$ in the above equation, we get 
$$D_3(\lambda 41^{\ell_4})^2=7a^6-35a^4b^2d+21a^2b^4d^2-b^6d^3,$$
 where $\ell_1=2\ell_4+\ell_3$.
We now divide both sides by $b^6$ and then multiply by $7^2D_3^3$ to arrive at 
$$Y^2=X^3-35dD_3X^2+147d^2D_3^2X-49d^3D_3^3,$$
where $X=7D_3a^2/b^2$ and $Y=7D_3\lambda 41^{\ell_4}/b^3$.
As in previous cases, we use \texttt{IntegralPoints} subroutine of MAGMA to compute all $\{17,59\}$-integral points on the above $24$ elliptic curves. Utilizing the same technique as before, we can conclude that none of these integer points lead to an integer solution of \eqref{eqp1}. 

\subsection*{Case III:  $p=5$} 
In this case, both $41$ and $59$ are primitive divisors of $\mathcal{L}_5(\alpha, \beta)$, and 
$(q, d)=(41, 1), (41, 59), (59, 1), (59, 17), (59, 41), (59, 1003)$. 
Equating the imaginary parts in \eqref{eqp2}, we get 
\begin{equation}\label{eqp6}
\lambda^217^{k_1}41^{\ell_1}59^{m_1}=\varepsilon b(5a^4-10a^2b^2d+b^4d^2),
\end{equation}
where $\varepsilon=\varepsilon_1\varepsilon_2=\pm 1$ . In case of $q=41$,  by the definition of primitive divisor, $b\ne 41$. Thus \eqref{eqp6} gives  
$$b=1,17^{k_1},59^{m_1}, 17^{k_1}59^{m_1}.$$
Similarly, when $q=59$,  we get (from \eqref{eqp6})
$$b=1,17^{k_1}, 41^{\ell_1}, 17^{k_1}41^{\ell_1}.$$

If $b=1$, then $d=1,17,41,59,1003$ and \eqref{eqp6} reduces to
$$\lambda^217^{k_1}41^{\ell_1}59^{m_1}=\varepsilon (5a^4-10a^2d+d^2).$$
Let $k_1=2k_4+k_3, \ell_1=2\ell_4+\ell_3, m_1=2m_4+m_3$ with $k_3,\ell_3, m_3\in\{0,1\}$ and set $D:=\varepsilon 17^{k_3}41^{\ell_3}59^{m_3}$. Then the above equation becomes 
$$D(\lambda 17^{k_4}41^{\ell_4}59^{m_4})^2=5a^4-10a^2d+d^2.$$
We now multiply both sides by $D$, and then put $(X, Y):=(a, D\lambda 17^{k_4}41^{\ell_4}59^{m_4})$ to get
$$Y^2=5DX^4-10DdX^2+Dd^2.$$
The above equation gives $80$ quartic curves each for each value of $Dd$. Here, we use \texttt{IntegralQuarticPoints} subroutine of MAGMA to compute all integral points on these curves. As in Case II, we first eliminate the integer points $(X,Y)$ such that $XY=0$ since $abd\ne0$. For the remaining points, we have $(X,Y, D)=(1,-2,-1), (2,41,41), (2,11,1),$ $ (12,149,1)$. Clearly, $(X,Y, D)=(1,-2,-1)$ gives $(x,y,\lambda,k,\ell,m)=(1,1,2,0,0,0)$. On the other hand, $(X,Y,D)=(2,41,41)$ gives $(a,\lambda, k,\ell, m)=(2,1,0,2,0)$. Here, $d=1$ and hence by \eqref{eqp3} $y=5$, which further implies $x=38$. Therefore $(x,y,\lambda, k, \ell, m, p)=(38,5,1,0,2,0,5)$ is a solution of \ref{eqp1}. Also, since $2,17,41,59$ are the only possible prime factors of $Y$, so that $(X,Y,D)=(2,11,1), (12,149,1)$ are not possible.

We now consider the case $b=17^{k_1}$.  Then \eqref{eqp6} becomes 
$$\lambda^241^{\ell_1}59^{m_1}=\varepsilon (5a^4-10a^2b^2d+b^4d^2).$$
Here, $d=1,17,41,59,1003$.
Let $D_1:=\varepsilon 41^{\ell_3}59^{m_3}$ for $k_3, \ell_3, m_3\in\{0, 1\}$, and divide both sides by $b^4$, we get
$$D_1(\lambda 41^{\ell_4}59^{m_4}/b^2)^2=5(a/b)^4-10(a/b)^2d+d^2,$$
 where $ \ell_1=2\ell_4+\ell_3, m_1=2m_4+m_3$.
We set $(X, Y):=(a/b, \lambda 41^{\ell_4}59^{m_4}/b^2)$ to arrive at the following:
$$D_1Y^2=5X^4-10dX^2+d^2.$$
We use \texttt{SIntegralLjunggrenPoints} subroutine of MAGMA to compute all $\{17\}$-integral points on the $40$ quartic curves defined by the above equation. As in the previous cases, we get $(X,Y, D)=(1,2,-1), (2,1,41), (2,11,1), (12,149,1)$. 
As before, these give $(x,y,\lambda,k,\ell,m,p)=(1,1,2,0,0,0,5), (38,5,1,0,2,0,5)$.

Assume that $b=41^{\ell_1}$.  Then $d=1, 59$ and \eqref{eqp6} reduces to
$$\lambda^217^{k_1}59^{m_1}=\varepsilon (5a^4-10a^2b^2d+b^4d^2).$$
Let $D_2:=\varepsilon 17^{k_3}59^{m_3}$ for $k_3, m_3\in\{0, 1\}$, and divide both sides by $b^4$, we get
$$D_2(\lambda 17^{k_4}59^{m_4}/b^2)^2=5(a/b)^4-10(a/b)^2d+d^2,$$
where $ k_1=2k_4+k_3, m_1=2m_4+m_3$.
Putting $(X, Y):=(a/b, \lambda 17^{k_4}59^{m_4}/b^2)$, we get
$$D_2Y^2=5X^4-10dX^2+d^2.$$
Here too we use \texttt{SIntegralLjunggrenPoints} subroutine of MAGMA to compute all $\{41\}$-integral points on the $16$ quartic curves defined by the last equation. As before, we get $(X,Y, D)=(1,2,-1), (12,149,1)$, which lead to  $(x,y,\lambda,k,\ell,m,p)=(1,1,2,0,0,0,5)$.

For $b=59^{m_1}$,  \eqref{eqp6} becomes 
$$\lambda^217^{k_1}41^{\ell_1}=\varepsilon (5a^4-10a^2b^2d+b^4d^2).$$
As in the previous cases, by setting $D_3:=\varepsilon 17^{k_3}41^{\ell_3}$, this equation can written as
$$D_3(\lambda 17^{k_4}41^{\ell_4}/b^2)^2=5(a/b)^4-10(a/b)^2d+d^2,$$
where $ k_1=2k_4+k_3, \ell_1=2\ell_4+\ell_3$.
Assume that $(X, Y):=(a/b, \lambda 17^{k_4}41^{\ell_4}/b^2)$. Then we get
$$D_3Y^2=5X^4-10dX^2+d^2.$$
As before, we get $(X,Y, D)=(1,2,-1), (2,1,41), (2,11,1)$. The solutions corresponding to these points are already discussed.

Finally for $b=41^{\ell_1}59^{m_1}$, \eqref{eqp6} implies to
$$\lambda^217^{k_1}=\varepsilon (5a^4-10a^2b^2d+b^4d^2).$$
We put $D_4:=\varepsilon 17^{k_3}$ for $k_3\in\{0, 1\}$, and divide both sides by $b^4$ to get
$$D_4(\lambda 17^{k_4}/b^2)^2=5(a/b)^4-10(a/b)^2d+d^2,$$
where $ k_1=2k_4+k_3$.
Writing $(X, Y):=(a/b, \lambda 17^{k_4}$, we get
$$D_3Y^2=5X^4-10dX^2+d^2.$$
Using \texttt{SIntegralLjunggrenPoints} subroutine of MAGMA, we compute all $\{41,59\}$-integral points on the $20$ quartic curves defined by the last equation. Here, we get $(X,Y, D)=(1,2,-1), (2,11,1), (12, 149,1)$, which are already treated in the previous cases. This completes the proof.
\end{proof}
\subsection{Proof of Theorem \ref{thm}}\label{proofthm}
The proof follows Propositions \ref{propm3} and \ref{propm4}, when $n$ is a divisibly by $3$ or $4$. For the remaining part, we assume that $n=pN$ for some prime $p\geq 5$ and an integer $N\geq 1$. Then \eqref{eqn1} can be written as
$$x^2+17^k41^\ell 59^m=\lambda\left(y^N\right)^p,~ x\geq 1, y>1, \gcd(x, y)=1, k,\ell, m\geq 0, N\geq 1, \lambda=1,2,4.
$$
By Proposition \ref{propp}, $(x,y,\lambda,k,\ell,m,p, N)=(38,5,1,0,2,0,5, 1)$ is the only integer solution of this equation. This completes the proof.  

\section{Concluding remarks}\label{S4}
Here, we discuss some existing results using our result and deduce some corollaries. We consider the Diophantine equation 
$$x^2+17^k=2^\delta y^n, x\geq 1, y>1, \gcd(x,y)=1, k, \delta, \geq 0, n\geq 3.$$
Clearly, it has no solutions when $\delta\geq 2$. For $\delta=0$, it follows from \cite[Theorem]{GW12} that $(x,y,k, n)=(8,3,1,4)$ is the only solution of the above equation. On the other hand for $\delta =2$, Abu Muriefah et al.  proved that its only  solutions are $(x,y,k,n)=(239,13,0,4), (31,5,2,4)$. One can put $\ell=m=0$ in Theorem \ref{thm} to get these results. 
On the other hand, \cite[Theorem 1.1]{AZ20} shows that the Diophantine equation  
$$x^2+41^\ell=2^\delta y^n, x\geq 1, y>1, \gcd(x,y)=1, \ell, \delta \geq 0, n\geq 3,$$
has no solutions when $\delta =0$, except for $(x,y,\ell, n)=(840,29,2,4), (38,5,2,5)$. Our result gives all solutions of this equation for any $\delta$. 

Finally we note down the following straightforward corollaries. 
\begin{cor}
The only solutions of the Diophantine equation  
$$x^2+59^m=2^\delta y^n, x\geq 1, y>1, \gcd(x,y)=1, m, \delta, \geq 0, n\geq 3,$$
are: $$(x,y,\delta, m, n)=(7,3,2,1,3), (21, 5,2,1,3), (525,41,2,1,3), (28735, 591, 2, 1,3), (239,13,1,0,4).$$ 
\end{cor}

\begin{cor}
The solutions of the Diophantine equation  
$$x^2+17^k41^{\ell}=2^\delta y^n, x\geq 1, y>1, \gcd(x,y)=1, k, \ell, \delta, \geq 0, n\geq 3,$$
are $(x,y,\delta,k,\ell, n)=(245,33, 1,2,1,3),  (840,29,0,0,2,4), (8,3,0,1,0,4), (239,13,1,0,0,4),\\ (11,3,1,0,1,4), (69,7,1,0,1,4), (171,11,1,0,1,4), (31,5,1,2,0,4), (299,15,1,2,1,4),\\ (38,5, 0,0,2,5)$.
\end{cor}
\begin{cor}
 The solutions of the Diophantine equation  
$$x^2+17^k59^m=2^\delta y^n, x\geq 1, y>1, \gcd(x,y)=1, k, m, \delta, \geq 0, n\geq 3$$
are: $(x,y,\delta,k, m, n)= (7,3,2,0,1,3), (21, 5, 2,0,1,3), (393,35,2,2,1,3), (525,41,2,0,1,3), \\ (5159, 189, 2, 2, 1, 3), (28735, 591, 2, 0, 1, 3), (721267, 6153, 2, 8, 1, 3), (8,3,0,1,0,0, 4),\\ (239,13,1,0,0,0, 4),  (31,5,1,2,0,0, 4)$.
\end{cor}
\begin{cor}
The Diophantine equation  
$$x^2+41^{\ell}59^m=2^\delta y^n, x\geq 1, y>1, \gcd(x,y)=1, \ell, m,\delta, \geq 0, n\geq 3,$$
has no solution, except:\\
$(x,y,\delta,\ell, m, n)=( 7,3,2,0,1,3), (21,5,2,0,1,3), (525,41,2,0,1,3), (10535, 411, 2, 4,1,3),\\  (55049, 945, 2, 2,3,3), (28735, 591, 2, 0, 1, 3), (840,29,0,2,0,4),  (239,13,1,0,0,4), \\
  (11,3,1,1,0, 4), (69,7,1,1,0, 4), (171,11,1,1,0, 4), (9,5,2,1,1, 4),  (38,5,0,2,0,5).$
\end{cor}

\subsection*{Acknowledgements}
The second author is grateful to Professor Kotyada Srinivas for stimulating environment at The Institute of Mathematical Sciences during his visiting period while working on this manuscript. The authors are grateful to Professor Douglas McNeil for pointing out a few misprints in the previous version. This project is supported by SERB MATRICS grant (No. MTR/2021/000762) and  CRG grant (No. CRG/2023/007323),  Govt. of India. 

\end{document}